\documentclass[11pt]{article}
\usepackage[T1]{fontenc}
\usepackage{lmodern,amsmath,amsthm,amsfonts,amssymb,graphicx,float,wrapfig,calc,microtype,thmtools,underscore,mathtools,paralist}
\usepackage[usenames,dvipsnames,svgnames,table]{xcolor}
\usepackage[T1]{fontenc}
\usepackage[unicode=true]{hyperref}
\hypersetup{
colorlinks,
linkcolor={black},
citecolor={black},
urlcolor={blue!60!black},
pdftitle={Clustered 3-Colouring Graphs of Bounded Degree},
pdfauthor={Vida Dujmovi{\'c}, Louis Esperet, Pat Morin, Bartosz Walczak, David~R.~Wood}}
\usepackage[noabbrev,capitalise]{cleveref}
\crefname{lem}{Lemma}{Lemmas}
\crefname{thm}{Theorem}{Theorems}
\crefname{cor}{Corollary}{Corollaries}
\crefname{prop}{Proposition}{Propositions}
\crefname{conj}{Conjecture}{Conjectures}
\crefname{openproblem}{Open Problem}{Open Problems}
\crefformat{equation}{(#2#1#3)}
\Crefformat{equation}{Equation #2(#1)#3}
\usepackage[numbers,sort&compress]{natbib}
\makeatletter
\def\NAT@spacechar{~}
\makeatother
\usepackage[bmargin=25mm,tmargin=25mm,lmargin=25mm,rmargin=25mm]{geometry}
\makeatletter
\def\thm@space@setup{\thm@preskip=\parskip\thm@postskip=0pt}
\makeatother
\let\oldproof\proof\let\endoldproof\endproof
\def\proof{\oldproof\unskip}\def\endproof{\endoldproof\unskip}

\allowdisplaybreaks

\DeclarePairedDelimiter{\floor}{\lfloor}{\rfloor}

\renewcommand{\le}{\leqslant}
\renewcommand{\geq}{\geqslant}
\renewcommand{\leq}{\leqslant}
\DeclareMathOperator{\dist}{dist}
\renewcommand{\thefootnote}{\fnsymbol{footnote}}
\allowdisplaybreaks
\theoremstyle{plain}
\newtheorem{thm}{Theorem}
\newtheorem{lem}[thm]{Lemma}
\newtheorem{cor}[thm]{Corollary}

\newtheorem{obs}[thm]{Observation}
\theoremstyle{definition}

\newcommand{\PP}{\mathcal{P}}
\newcommand{\QQ}{\mathcal{Q}}
\newcommand{\N}{\mathbb{N}}

\newcommand{\NN}{\mathbb{N}_0}

\newcommand{\defn}[1]{\textcolor{Maroon}{\emph{#1}}}

\def\longequation{$$\vcenter\bgroup\advance\hsize by -9em%
\noindent\ignorespaces\refstepcounter{equation}}%
\makeatletter%
\def\endlongequation{\egroup\eqno(\theequation)$$\global\@ignoretrue}
\makeatother

\begin{document}

\author{Vida Dujmovi{\'c}\,\footnotemark[1]
\qquad Louis Esperet\,\footnotemark[2]
\qquad Pat Morin\,\footnotemark[6]\\
\qquad Bartosz Walczak\,\footnotemark[4]
\qquad David~R.~Wood\,\footnotemark[5]}

\date{}

\footnotetext[1]{School of Computer Science and Electrical Engineering, University of Ottawa, Ottawa, Canada (\texttt{vida.dujmovic@uottawa.ca}). Research supported by NSERC and the Ontario Ministry of Research and Innovation.}

\footnotetext[2]{Laboratoire G-SCOP (CNRS, Univ.\ Grenoble Alpes), Grenoble, France
(\texttt{louis.esperet@grenoble-inp.fr}). Partially supported by ANR Projects GATO
(\textsc{anr-16-ce40-0009-01}) and GrR (\textsc{anr-18-ce40-0032}).}

\footnotetext[6]{School of Computer Science, Carleton University, Ottawa, Canada (\texttt{morin@scs.carleton.ca}). Research supported by NSERC.}

\footnotetext[4]{Department of Theoretical Computer Science, Faculty of Mathematics and Computer Science, Jagiellonian University, Krak\'ow, Poland (\texttt{walczak@tcs.uj.edu.pl}). Research partially supported by National Science Centre of Poland grant 2015/17/D/ST1/00585.}

\footnotetext[5]{School of Mathematics, Monash University, Melbourne, Australia (\texttt{david.wood@monash.edu}). Research supported by the Australian Research Council.}

\sloppy

\title{\textbf{Clustered 3-Colouring Graphs\\
of Bounded Degree}}

\maketitle


\begin{abstract}
A (not necessarily proper) vertex colouring of a graph has \defn{clustering} $c$ if every monochromatic component has at most $c$ vertices. We prove that planar graphs with maximum degree $\Delta$ are 3-colourable with clustering $O(\Delta^2)$. The previous best bound was $O(\Delta^{37})$. This result for planar graphs generalises to graphs that can be drawn on a surface of bounded Euler genus with a bounded number of crossings per edge. We then prove that graphs with maximum degree $\Delta$ that exclude a fixed minor are 3-colourable with clustering $O(\Delta^5)$. The best previous bound for this result was exponential in $\Delta$.
\end{abstract}

\renewcommand{\thefootnote}{\arabic{footnote}}

\section{Introduction}
\label{Introduction}

Consider a graph where each vertex is assigned a colour. A \defn{monochromatic component} is a connected component of the subgraph induced by all the vertices assigned a single colour. A graph $G$ is $k$-colourable with \defn{clustering} $c$ if each vertex can be assigned one of $k$ colours so that each monochromatic component has at most $c$ vertices. There have been several recent papers on clustered colouring \citep{NSSW19,vdHW18,KO19,CE19,EJ14,HST03,EO16,DN17,LO17,HW19,MRW17,LW1,LW2,LW3,NSW}; see \citep{WoodSurvey} for a survey. The general goal of this paper is to prove that various classes of graphs are 3-colourable with clustering bounded by a polynomial function of the maximum degree.

First consider clustered colouring of planar graphs. The 4-colour theorem \citep{AH89,RSST97} says that every planar graph is 4-colourable with clustering 1. This result is best possible regardless of the clustering value: for every integer $c$ there is a planar graph $G$ such that every 3-colouring of $G$ has a monochromatic component with more than $c$ vertices \citep{WoodSurvey,ADOV03,EJ14,KMRV97}. All known examples of such graphs have unbounded maximum degree. This led \citet*{KMRV97} to ask whether planar graphs with bounded maximum degree are 3-colourable with bounded clustering. This question was answered positively by \citet*{EJ14}.

Three colours is best possible for $\Delta\geq 6$, since the Hex Lemma~\citep{Gale79} implies that for every integer $c$, there is a planar graph $G$ with maximum degree 6 such that every 2-colouring of $G$ has a monochromatic component with more than $c$ vertices \citep{LMST08,MP08}. Furthermore, this degree threshold is best possible, since \citet*{HST03} proved that every graph with maximum degree 5 (regardless of planarity) is 2-colourable with clustering less than 20,000.

The following natural question arises: what is the least function $c(\Delta)$ such that every planar graph with maximum degree $\Delta$ has a 3-colouring with clustering $c(\Delta)$? The clustering function of \citet*{EJ14} was $\Delta^{O(\Delta)}$. While \citet*{EJ14} made no effort to optimise this function, exponential dependence on $\Delta$ is unavoidable using their method. Recently, \citet*{LW1} improved this bound to $O(\Delta^{37})$. A primary contribution of this paper  (\cref{3ColourPlanar}) is to improve it further to $O(\Delta^2)$. 

Like the above-mentioned works of \citet*{EJ14} and \citet*{LW1}, our theorem generalises to graphs with bounded Euler genus\footnote{The \textit{Euler genus} of the orientable surface with $h$ handles is $2h$. The \textit{Euler genus} of the non-orientable surface with $c$ cross-caps is $c$. The \textit{Euler genus} of a graph $G$ is the minimum integer $k$ such that $G$ embeds in a surface of Euler genus $k$. Of course, a graph is planar if and only if it has Euler genus 0; see \citep{MoharThom} for more about graph embeddings in surfaces.\newline
A graph $H$ is a \textit{minor} of a graph $G$ if a graph isomorphic to $H$ can be obtained from a subgraph of $G$ by contracting edges. A class $\mathcal{G}$ of graphs is \defn{minor-closed} if for every graph $G\in\mathcal{G}$, every minor of $G$ is in $\mathcal{G}$.
A minor-closed class is \defn{proper} if it is not the class of all graphs. For example, for fixed $g\geq 0$, the class of graphs with Euler genus at most $g$ is a proper minor-closed class.\newline
A graph $H$ is \defn{apex} if $H-v$ is planar for some vertex $v$.}. In particular, we prove  (in \cref{3ColourGenus}) that graphs with Euler genus $g$ and maximum degree $\Delta$ are 3-colourable with clustering $O(g^3\Delta^2)$. The previous best clustering function was $O(g^{19}\Delta^{37})$ due to \citet*{LW1}. In fact, our result and that of \citet*{LW1} hold in the more general setting of bounded layered treewidth (defined in \cref{LayeredTreewidth}). This enables further generalisations. For example, we prove (in \cref{3ColourApex}) that apex-minor-free graphs are 3-colourable with clustering $O(\Delta^2)$, and graphs that have a drawing on a surface of bounded Euler genus with a bounded number of crossings per edge are 3-colourable with clustering $O(\Delta^2)$. All these results are presented in \cref{Planar}.

\cref{Minors} focuses on clustered colouring of graphs excluding a fixed minor. For $K_t$-minor-free graphs, at least $t-1$ colours are needed regardless of the clustering function; that is, for every integer $c$ there is a $K_t$-minor-free graph $G$ such that every $(t-2)$-colouring of $G$ has a monochromatic component with more than $c$ vertices \citep{WoodSurvey,EKKOS15}. Again, all such examples have unbounded maximum degree. Indeed, in the setting of bounded degree graphs, qualitatively different behaviour occurs. In particular, \citet*{LO17} proved that bounded degree graphs excluding a fixed minor are 3-colourable with bounded clustering (thus generalising the above result of \citet*{EJ14} for planar graphs and graphs of bounded Euler genus).

\citet*{LO17} did not state an explicit bound on the clustering function, but it is at least exponential in the maximum degree\footnote{Chun-Hung Liu [private communication, 2020] believes that the method in \citep{LO17} could be adapted to give a polynomial bound using more advanced graph structure theorems.}. We prove (in \cref{3colMinor}) that graphs with maximum degree $\Delta$ that exclude a fixed minor are 3-colourable with clustering $O(\Delta^5)$. The proof of this result is much simpler than that of \citet*{LO17}, and is based on a new structural description of bounded-degree graphs excluding a minor that is of independent interest (\cref{MinorFreeDeltaLayeredPartition,MinorFreeDegreeStructure}).

Bounded maximum degree alone is not enough to ensure an absolute bound (independent of the degree) on the number of colours in a clustered colouring. In particular, for all integers $\Delta\geq 2$ and $c$ there is a graph $G$ with maximum degree $\Delta$ such that every $\floor{\frac{\Delta+2}{4}}$-colouring of $G$ has a monochromatic component with more than $c$ vertices; see \citep{ADOV03,HST03,WoodSurvey}. This says that in all of the above results, to achieve an absolute bound on the number of colours, one must assume some structural property (such as bounded treewidth, being planar, or excluding a minor) in addition to assuming bounded maximum degree.

To conclude our literature survey, we mention the results of \citet*{LW1,LW2,LW3} that generalise the bounded degree setting. First, \citet*{LW1} proved that for all $s,t,k\in\N$ there exists $c\in\N$ such that every graph with layered treewidth $k$ and with no $K_{s,t}$ subgraph is $(s+2)$-colourable with clustering $c$. The case $s=1$ is equivalent to the bounded degree setting; thus this result generalises the above-mentioned 3-colouring results for graphs with bounded maximum degree. For $s\geq 2$, the clustering function here is very large, and the proof is 70+ pages long. In the setting of excluded minors, \citet*{LW2} proved that for all $s,t\in\N$ and for every graph $H$ there is an integer $c$ such that every graph with no $H$-minor and with no $K_{s,t}$-subgraph is $(s+2)$-colourable with clustering $c$. Similar results are obtained for excluded topological minors \citep{LW3}. 

\section{Planar Graphs and Generalisations}
\label{Planar}

This section proves that planar graphs with maximum degree $\Delta$ (and other more general classes) are 3-colourable with clustering $O(\Delta^2)$. Let $\N:=\{1,2,\dots\}$ and $\NN:=\{0,1,\dots\}$. 

\subsection{Treewidth}

Tree decompositions and treewidth are used throughout this paper.
For two graphs $G$ and $H$, an \defn{$H$-decomposition} of $G$ consists of a collection $(B_x : x\in V(H))$ of subsets of $V(G)$, called \defn{bags}, indexed by the nodes of $H$, such that:
\begin{compactitem}
\item for every vertex $v$ of $G$, the set $\{x\in V(H) : v\in B_x\}$ induces a non-empty connected subgraph of $H$, and
\item for every edge $vw$ of $G$, there is a vertex $x\in V(H)$ for which $v,w\in B_x$.
\end{compactitem}
The \defn{width} of such an $H$-decomposition is $\max\{|B_x|:x\in V(H)\}-1$. A \defn{tree-decomposition} is a $T$-decomposition for some tree $T$. Tree decompositions were introduced by \citet*{Halin76} and \citet*{RS-II}. The more general notion of $H$-decomposition was introduced by \citet*{DK05}. The \defn{treewidth} of a graph $G$ is the minimum width of a tree-decomposition of $G$. Treewidth measures how similar a given graph is to a tree. It is particularly important in structural and algorithmic graph theory; see \citep{HW17,Reed97,Bodlaender-TCS98} for surveys.

Our first tool, which was also used by \citet*{LW1}, is the following 2-colouring result for graphs of bounded treewidth due to \citet*{ADOV03}. The constant 20 comes from applying a result from \citep{Wood09}.

\begin{lem}[\citep{ADOV03}]
\label{2Colour}
Every graph with maximum degree\/ $\Delta\in\N$ and treewidth less than\/ $k\in\N$ is\/ $2$-colourable with clustering\/ $20k\Delta$.
\end{lem}

As an aside, it follows from the Lipton--Tarjan separator theorem~\citep{LT79} that $n$-vertex planar graphs have treewidth $O(\sqrt{n})$. Thus \cref{2Colour} implies that $n$-vertex planar graphs with maximum degree $\Delta\in\N$ are 2-colourable with clustering $O(\Delta\sqrt{n})$, which answers an open problem raised by \citet*{LMST08}. The same result holds for graphs excluding any fixed minor, using the separator theorem of \citet*{AST90}.

\subsection{Key Lemma}

The next lemma is a central result of the paper. Here, a \defn{layering} of a graph $G$ is an ordered partition $(V_0,V_1,\dots)$ of $V(G)$ such that for every edge $vw\in E(G)$, if $v\in V_i$ and $w\in V_j$, then $|i-j| \leq 1$. For example, if $r$ is a vertex in a connected graph $G$ and $V_i:=\{v\in V(G):\dist_G(r,v)=i\}$ for all $i\in\NN$, then $(V_0,V_1,\dots)$ is called a \textit{BFS layering} of $G$. The lemma assumes that for some layering of a graph, every subgraph induced by a bounded number of consecutive layers has bounded treewidth. This property dates to the seminal work of \citet*{Baker94}, who used it to obtain efficient approximation algorithms for various NP-hard problems on planar graphs. We show that graphs that satisfy this property and have small maximum degree are 3-colourable with small clustering.

\begin{lem}
\label{3Colour}
Let\/ $G$ be a graph with maximum degree\/ $\Delta\in\N$.
Let\/ $(V_0,V_1,\dots)$ be a layering of\/ $G$ such that\/ $G[\bigcup_{j=0}^{10}V_{i+j}]$ has treewidth less than\/ $k\in\N$ for all\/ $i\in\NN$. Then\/ $G$ is\/ $3$-colourable with clustering\/ $8000 k^3\Delta^2$.
\end{lem}

\begin{proof}
No attempt is made to improve the constant 8000. 
We may assume (by renaming the layers) that $V_0=V_1=V_2=V_3=V_4=\emptyset$. 

Let $\overline{i}:=i\bmod{3}$ for $i\in\NN$.
As illustrated in \cref{ProofIllustration}, for $i\in\NN$, let $G_i$ be the induced subgraph $G[V_{6i}\cup V_{6i+1}\cup \dots \cup V_{6i+4}]$. Thus $G_i$ has maximum degree at most $\Delta$ and treewidth less than $k$. By \cref{2Colour}, $G_i$ has a 2-colouring $c_i$ with clustering $20k\Delta$. Use colours $\overline{i}$ and $\overline{i+1}$ for this colouring of $G_i$. We now define the desired colouring $c$ of $G$. Vertices in $V_{6i}\cup V_{6i+1}$ coloured $\overline{i}$ in $c_i$ keep this colour in $c$.
Vertices in $V_{6i+2}$ keep their colour from $c_i$ in $c$.
Vertices in $V_{6i+3}\cup V_{6i+4}$ coloured $\overline{i+1}$ in $c_i$ keep this colour in $c$.
Other vertices are assigned a new colour, as we now explain.

\begin{figure}[!t]
\centering\includegraphics[width=\textwidth]{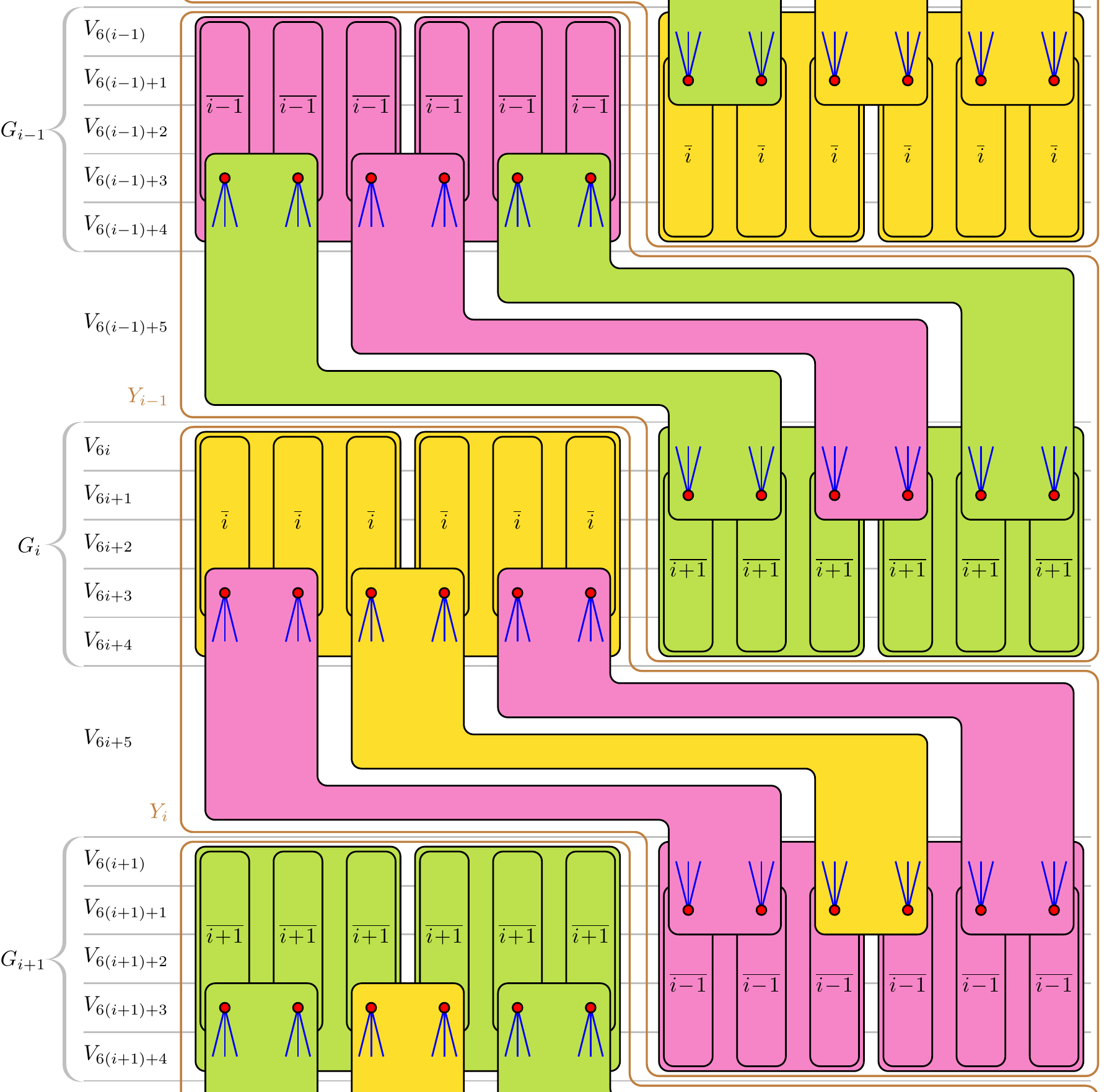}
\caption{\label{ProofIllustration} Proof of \cref{3Colour}.}
\end{figure}

For $j\in\NN$ and $\ell\in\{0,1,2\}$, let $V_{j,\ell}$ be the set of vertices in $V_j$ coloured $\ell$ in the colouring $c_i$ of the corresponding graph $G_i$ (which is well defined since $G_0,G_1,\dots$ are pairwise disjoint). For $i\in\NN$, let
\begin{align*}
A_i := \bigcup_{j=0}^3 V_{6i+j,\overline{i}}\quad, \quad
B_i := \bigcup_{j=1}^4 V_{6(i+1)+j,\overline{i-1}}\quad, \quad
Y_i := A_i \,\cup\, V_{6i+4,\overline{i}} \,\cup\, V_{6i+5} \,\cup\, V_{6(i+1),\overline{i-1}} \,\cup\, B_i.
\end{align*}
Note that $\{Y_i:i\in\NN\}$ is a partition of $V(G)$ (since $V_0=V_1=V_2=V_3=V_4=\emptyset$). 
In fact, $(Y_0,Y_1,\dots)$ is a layering of $G$, since $V_{6i-1}$ separates $Y_0\cup \dots\cup Y_{i-2}$ and $Y_i\cup Y_{i+1} \cup \cdots$.

For $i\in\NN$, let $Z_i$ be the graph obtained from $G[Y_i]$ as follows: for each component $X$ of $G[A_i]$ or $G[B_i]$, contract $X$ into a single vertex $v_X$. The neighbours of $v_X$ in $Z_i$ are contained within a monochromatic component of $G_i$ or $G_{i+1}$; thus $v_X$ has degree at most the size of the corresponding monochromatic component of $G_i$ or $G_{i+1}$, which is at most $20k\Delta$. Since $Z_i$ is a minor of $G[\bigcup_{j=0}^{10}V_{6i+j}]$ and treewidth is a minor-monotone parameter, $Z_i$ has treewidth less than $k$. By \cref{2Colour}, $Z_i$ has a 2-colouring $c_i'$ with clustering $400 k^2\Delta$. Use colours $\overline{i}$ and $\overline{i-1}$ for this colouring of $Z_i$.

We now assign colours to the remaining vertices of $G$ in the colouring $c$. Vertices in $V_{6i+4,\overline{i}} \cup V_{6i+5} \cup V_{6(i+1),\overline{i-1}}$ keep their colour from the colouring $c_i'$ of $Z_i$. Note that these vertices were not contracted in the construction of $Z_i$. For each component $X$ of $G[A_i]$, assign the colour given to $v_X$ in $c_i'$ to each vertex in $X\cap V_{6i+3}$. Similarly, for each component $X$ of $G[B_i]$, assign the colour given to $v_X$ in $c_i'$ to each vertex in $X\cap V_{6i+7}$. This completes the definition of the colouring $c$ of $G$.

Consider a monochromatic component $M$ in the 3-colouring $c$ of $G$. Suppose that $M$ contains an edge $vw$ with $v\in Y_{i-1}$ and $w\in Y_i$ for some $i\in\NN$. The only colour used by both $Y_{i-1}$ and $Y_i$ is $\overline{i-1}$; thus $M$ is coloured $\overline{i-1}$. But $V_{6i+2}$ does not use colour $\overline{i-1}$, and it separates $Y_{i-1}$ and $Y_i$. This contradiction shows that $M$ contains no such edge $vw$. Since $(Y_0,Y_1,\dots)$ is a layering of $G$ and $M$ is connected, $M$ is contained in some $Y_i$. The only colours used in $Y_i$ are $\overline{i}$ and $\overline{i-1}$. By symmetry we may assume that $M$ is coloured $\overline{i}$.

If $M$ is contained in $V_{6i} \cup V_{6i+1} \cup V_{6i+2}$, then $M$ is contained in some monochromatic component of $G_i$ (with respect to the colouring $c_i$), and thus $|V(M)| \leq 20k\Delta$. Otherwise, $M$ is contained in the graph obtained from a monochromatic component $C$ of $Z_i$ (with respect to the colouring $c_i'$) by replacing each contracted vertex $v_X$ in $C$ by $X$. Since $|V(C)| \leq 400 k^2\Delta$ and $|V(X)| \leq 20k\Delta$, we conclude that $|V(M)| \leq 8000 k^3\Delta^2$.
\end{proof}

See Appendix~A for a slightly stronger and slightly simpler version of \cref{3Colour} that was added after the paper was accepted to \emph{Combinatorics, Probability \& Computing}. 

\subsection{Layered Treewidth}
\label{LayeredTreewidth}

\citet*{DMW17} and \citet*{Shahrokhi13} independently introduced the following concept. The \defn{layered treewidth} of a graph $G$ is the minimum integer $k$ such that $G$ has a tree-decomposition $(B_x:x\in V(T))$ and a layering $(V_0,V_1,\dots)$ such that $|B_x\cap V_i|\leq k$ for every bag $B_x$ and layer $V_i$. Applications of layered treewidth include graph colouring \citep{DMW17,LW1,vdHW18}, graph drawing \citep{DMW17,BDDEW18}, book embeddings \citep{DF18}, boxicity \citep{SW20}, and intersection graph theory \citep{Shahrokhi13}. The related notion of layered pathwidth has also been studied \citep{DEJMW20,BDDEW18}. In a graph with layered treewidth $k$, the subgraph induced by the union of any 11 consecutive layers has treewidth less than $11k$. Thus \cref{3Colour} implies:

\begin{cor}
\label{3ColourLayeredTreewidth}
Every graph with layered treewidth\/ $k\in\N$ and maximum degree\/ $\Delta\in\N$ is\/ $3$-colourable with clustering\/ $O(k^3\Delta^2)$.
\end{cor}

This corollary improves on a result of \citet*{LW1} who proved an upper bound of $O(k^{19}\Delta^{37})$ on the clustering function.

Many classes of graphs are known to have bounded layered treewidth. For example, \citet*{DMW17} proved that every planar graph has layered treewidth at most 3, every graph with Euler genus $g$ has layered treewidth at most $2g+3$, and that any apex-minor-free class of graphs has bounded layered treewidth. \cref{3ColourLayeredTreewidth} thus implies the following results.

\begin{cor}
\label{3ColourPlanar}
Every planar graph with maximum degree\/ $\Delta\in\N$ is\/ $3$-colourable with clustering\/ $O(\Delta^2)$.
\end{cor}

\begin{cor}
\label{3ColourGenus}
Every graph with Euler genus\/ $g\in\NN$ and maximum degree\/ $\Delta\in\N$ is\/ $3$-colourable with clustering\/ $O(g^3\Delta^2)$.
\end{cor}

\begin{cor}
\label{3ColourApex}
For every fixed apex graph\/ $H$, every\/ $H$-minor-free graph with maximum degree\/ $\Delta\in\N$ is\/ $3$-colourable with clustering\/ $O(\Delta^2)$.
\end{cor}

The above corollaries can also be deduced from \cref{3Colour} without considering layered treewidth. First consider a planar graph $G$, which we may assume is connected. Let $(V_0,V_1,\dots)$ be a BFS layering of $G$. For $i\in\NN$, let $G_i$ be obtained from $G[V_0\cup V_1\cup\dots\cup V_{i+10}]$ by contracting $G[V_0\cup \dots\cup V_{i-1}]$ (which is connected) into a single vertex.
Thus $G_i$ is planar and has radius at most 11.
\citet*{RS-III} proved that every planar graph with radius $d$ has treewidth at most $3d$.
Thus $G[V_i\cup \dots \cup V_{i+10}]$, which is a subgraph of $G_i$, has treewidth at most 33.
\cref{3ColourPlanar} then follows from \cref{3Colour}.
The same proof works in any minor-closed class for which the treewidth of any graph $G$ in the class is bounded by a function of the radius of $G$. For example, \citet*{Eppstein-Algo00} proved that every graph with Euler genus $g$ and radius $d$ has treewidth at most $O(gd)$. \cref{3ColourGenus} follows.
More generally, \citet*{Eppstein-Algo00} proved that for every apex graph $H$,
every $H$-minor-free graph with bounded radius has bounded treewidth. \cref{3ColourApex} follows.

Finally, note that one can also prove that every graph with Euler genus $g$ and maximum degree $\Delta$ is 3-colourable with clustering $O(g\Delta^6)$ using \cref{3Colour} and a result of \citet*{EJ14}\footnote{\citet*{EJ14} proved that if every plane graph with maximum degree $\Delta$ has a 3-colouring with clustering $f(\Delta)$, where one colour is not used on the outerface, then graphs with Euler genus $g$ and maximum degree $\Delta$ are 3-colourable with clustering $O(\Delta^2 f(\Delta)^2 g)$. Now, let $G$ be a plane graph with maximum degree $\Delta$. Let $G^+$ be the plane graph obtained by adding one new vertex $r$ adjacent to the vertices on the outerface of $G$. For $i\in\NN$, let $V_i$ be the set of vertices in $G^+$ at distance $i$ from $r$ in $G^+$. By the above contraction argument, $(V_1,\dots,V_n)$ is a layering of $G$ such that any set of 11 consecutive layers induces a subgraph with bounded treewidth. By \cref{3Colour}, $G$ is 3-colourable with clustering $O(\Delta^2)$. Moreover, only two colours are used on $V_1$ and thus on the outerface of $G$. By the above-mentioned result of \citet*{EJ14} with $f(\Delta)=O(\Delta^2)$, all graphs with Euler genus $g$ and maximum degree $\Delta$ are 3-colourable with clustering $O(\Delta^6 g)$.}.

\subsection{Examples}

One advantage of considering layered treewidth is that several non-minor-closed classes of interest have bounded layered treewidth. We give three examples:

\textbf{\boldmath $(g,k)$-Planar Graphs:}
A graph is \defn{$(g,k)$-planar} if it has a drawing on a surface of Euler genus at most $g$ such that each edge is involved in at most $k$ crossings (with other edges). \citet*{DEW17} proved that every $(g,k)$-planar graph has layered treewidth $O(gk)$. \cref{3ColourLayeredTreewidth} implies
that every $(g,k)$-planar graph with maximum degree $\Delta$ is 3-colourable with clustering $O(g^3k^3\Delta^2)$. This improves on a result of \citet*{LW1} who proved an upper bound of $O(g^{19} k^{19}\Delta^{37})$ on the clustering function.

\textbf{Map Graphs:}
Map graphs are defined as follows. Start with a graph $G_0$ embedded in a surface of Euler genus $g$, with each face labelled a `nation' or a `lake', where each vertex of $G_0$ is incident with at most $d$ nations. Let $G$ be the graph whose vertices are the nations of $G_0$, where two vertices are adjacent in $G$ if the corresponding faces in $G_0$ share a vertex. Then $G$ is called a \defn{$(g,d)$-map graph}. A $(0,d)$-map graph is called a (plane) \defn{$d$-map graph}; see \citep{FLS-SODA12,CGP02} for example.
The $(g,3)$-map graphs are precisely the graphs of Euler genus at most $g$; see \citep{DEW17}. So $(g,d)$-map graphs generalise graphs embedded in a surface.
\citet*{DEW17} showed that every $(g,d)$-map graph has layered treewidth at most $(2g+3)(2d+1)$.
\cref{3ColourLayeredTreewidth} then implies that every $(g,d)$-map graph with maximum degree $\Delta$ is 3-colourable with clustering $O(g^3d^3\Delta^2)$. This improves on a result of \citet*{LW1} who proved an upper bound of $O(g^{19} d^{19}\Delta^{37})$ on the clustering function.

\textbf{Graph Powers:}
For $p\in\N$, the \defn{$p$-th power} of a graph $G$ is the graph $G^p$ with vertex set $V(G^p):=V(G)$, where $vw\in E(G^p)$ if and only if $\dist_G(v,w)\leq p$. It follows from the work of \citet*{DMW19b} that powers of graphs with bounded layered treewidth and bounded maximum degree have bounded layered treewidth. Here we give a direct proof with better bounds.

\begin{lem}
\label{PowerLayeredTreewidth}
If\/ $G$ is a graph with layered treewidth\/ $k\in\N$ and maximum degree\/ $\Delta\in\N$, then\/ $G^p$ has layered treewidth less than\/ $2pk\Delta^{\floor{p/2}}$.
\end{lem}

\begin{proof}
The result is trivial if $\Delta=1$, so assume that $\Delta\geq 2$. Let $(V_1,V_2,\dots)$ be a layering of $G$ and let $(B_x: x\in V(T))$ be a tree-decomposition of $G$ such that $|V_i\cap B_x| \leq k$ for each $i\in\NN$ and $x\in V(T)$. For each vertex $v\in V(G)$, let $X_v:=\{ w\in V(G): \dist_G(v,w) \leq \floor{\frac{p}{2}}\}$. For each node $x\in V(T)$, let $B'_x:=\bigcup_{v\in B_x} X_v$.

We now prove that $(B'_x:x\in V(T))$ is a tree-decomposition of $G^p$. Consider a vertex $\alpha\in V(G^p)$. Since $\alpha\in X_v$ if and only if $v\in X_\alpha$, $$\{x\in V(T): \alpha \in B'_x\} = \bigcup_{v\in X_\alpha} \{ x\in V(T): v \in B_x \}.$$
Since $\{ x\in V(T): v \in B_x \}$ induces a connected subtree of $T$, and $X_\alpha$ induces a connected subgraph of $G$, it follows that $\{x\in V(T): \alpha \in B'_x\}$ also induces a connected subtree of $T$.
Now, consider an edge $\alpha\beta\in E(G^p)$. There is an edge $vw$ of $G$ (in the `middle' of a shortest $\alpha\beta$-path) such that $\alpha\in X_v$ and $\beta\in X_w$. Now $v,w\in B_x$ for some node $x\in V(T)$. By construction, $\alpha,\beta\in B'_x$. This shows that $(B'_x:x\in V(T))$ is a tree-decomposition of $G^p$.

For $i\in\NN$, let $W_i:= V_{ip}\cup V_{ip+1}\cup\dots\cup V_{(i+1)p-1}$. For each edge $\alpha\beta\in E(G^p)$, if $\alpha\in V_i$ and $\beta \in V_j$, then $|i-j|\leq p$. Thus if $\alpha\in W_{i'}$ and $\beta\in W_{j'}$, then $|i'-j'|\leq 1$. This shows that $(W_1,W_2,\dots)$ is a layering of $G^p$.
Since $|X_v| < 2\Delta^{\floor{p/2}}$ for each vertex $v\in V(G)$, for each node $x\in V(T)$ and $i\in\NN$, we have $|B'_x \cap V_i | < 2 k\Delta^{\floor{p/2}}$, implying $|B'_x \cap W_i | < 2 pk\Delta^{\floor{p/2}}$.
Therefore $G^p$ has layered treewidth less than $2pk\Delta^{\floor{p/2}}$.
\end{proof}

\cref{3ColourLayeredTreewidth,PowerLayeredTreewidth} imply that for every graph with layered treewidth $k$ and maximum degree $\Delta$, the $p$-th power $G^p$ (which has maximum degree less than $2\Delta^p$) is 3-colourable with clustering $O( k^3 \Delta^{3\floor{p/2} + 2p})$. For example, for every $(g,k)$-planar graph $G$ with maximum degree $\Delta$, the $p$-th power $G^p$ has a 3-colouring with clustering $O( g^3k^3 \Delta^{3\floor{p/2} + 2p})$.

\section{Excluded Minors}
\label{Minors}

This section shows that graphs excluding a fixed minor and with maximum degree $\Delta$ are 3-colourable with clustering $O(\Delta^5)$. The starting point is Robertson and Seymour's Graph Minor Structure Theorem, which we now introduce.

\subsection{Graph Minor Structure Theorem}

For a graph $G_0$ embedded in a surface, and a facial cycle $F$ of $G_0$ (thought of as a subgraph of $G_0$), an \defn{$F$-vortex} (relative to $G_0$) is an $F$-decomposition $(B_x\subseteq V(H):x\in V(F))$ of a graph $H$ such that $V(G_0\cap H)=V(F)$ and $x\in B_x$ for each $x\in V(F)$. 

For $k\in\NN$, a graph $G$ is \defn{$k$-almost embeddable} if for some set $A\subseteq V(G)$ with $|A|\leq k$ and for some $s\in\{0,\dots,k\}$, there are graphs $G_0,G_1,\dots,G_s$ such that:
\begin{compactitem}
\item $G-A = G_{0} \cup G_{1} \cup \cdots \cup G_s$,
\item $G_{1}, \dots, G_s$ are pairwise vertex-disjoint;
\item $G_{0}$ is embedded in a surface of Euler genus at most $k$,
\item there are $s$ pairwise vertex-disjoint facial cycles $F_1,\dots,F_s$ of $G_0$, and
\item for $i\in\{1,\dots,s\}$, there is an $F_i$-vortex $(B_x\subseteq V(G_i):x\in V(F_i))$ of $G_i$ (relative to $G_0$) of width at most $k$.
\end{compactitem}
The vertices in $A$ are called \defn{apex} vertices. They can be adjacent to any vertex in $G$.

It is not clear whether the class of $k$-almost embeddable graphs is hereditary, so it will be convenient to define a graph to be \defn{$k$-almost\/$^{\downarrow}\!$ embeddable} if it is an induced subgraph of some \defn{$k$-almost embeddable} graph.

In a tree-decomposition $(B_x : x\in V(T))$ of a graph $G$, the \defn{torso} of a bag $B_x$ is the graph obtained from $G[B_x]$ as follows: for every edge $xy\in E(T)$, add every edge $vw$ where $v,w\in B_x\cap B_y$.

The following graph minor structure theorem by \citet*{RS-XVI} is at the heart of graph minor theory.

\begin{thm}[\citep{RS-XVI}]
\label{GMST}
For every graph\/ $H$, there exists\/ $k\in\NN$ such that every graph\/ $G$ that does not contain\/ $H$ as a minor has a tree decomposition\/ $(B_x : x\in V(T))$ such that the torso\/ $G_x$ of\/ $B_x$ is\/ $k$-almost embeddable for each node\/ $x\in V(T)$.
\end{thm}

In \cref{GMST}, we have $|B_x\cap B_y| \leq 8k$ for each edge $xy$ of $T$ because of the following lemma.

\begin{lem}[{\protect\citep[Lemma~21]{DMW17}}]
\label{CliqueSize}
Every clique in a\/ $k$-almost embeddable graph has size at most\/ $8k$.
\end{lem}

We need the following slight strengthening of \cref{GMST}.

\begin{thm}
\label{GMST2}
For every graph\/ $H$, there exists\/ $k\in\N$ such that every graph\/ $G$ that does not contain\/ $H$ as a minor and has maximum degree at most\/ $\Delta\in\N$ has a tree decomposition\/ $(B_x : x\in V(T))$ such that for each node\/ $x\in V(T)$, the torso\/ $G_x$ of\/ $B_x$ is\/ $k$-almost\/$^{\downarrow}\!$ embeddable and has maximum degree less than\/ $8 k\Delta$.
\end{thm}

\begin{proof}
Let $(B_x:x\in V(T))$ be a tree decomposition of $G$ such that each torso is $k$-almost$^{\downarrow}\!$ embeddable, and subject to this condition, $\sum_{x\in V(T)}|B_x|$ is minimum. This is well defined by \cref{GMST}.

Consider an edge $xy\in E(T)$. Let $T_{x,y}$ be the component of $T-xy$ containing $x$. Let $V_{x,y}:= \bigcup\{B_z\setminus B_y :z\in V(T_{x,y})\}$. Suppose for the sake of contradiction that some vertex $v\in B_x\cap B_y$ has no neighbour in $V_{y,x}$. Let $B'_z:= B_z\setminus\{v\}$ for each $z\in V(T_{y,x})$, and let $B'_z:=B_z$ for each $z\in V(T_{x,y})$. Since induced subgraphs of $k$-almost$^{\downarrow}\!$ embeddable graphs are $k$-almost$^{\downarrow}\!$ embeddable, $(B'_z:z\in V(T))$ is a tree decomposition of $G$ such that each torso is $k$-almost$^{\downarrow}\!$ embeddable. (This is the reason we define $k$-almost$^{\downarrow}\!$ embeddability.)\ Since $v\in B_y$, we have $|B'_y| <|B_y|$, implying $\sum_{z\in V(T)}|B'_z| < \sum_{z\in V(T)}|B_z|$. This contradicts the choice of $(B_x:x\in V(T))$. Hence every vertex in $B_x\cap B_y$ has a neighbour in $V_{y,x}$.

Consider a node $x\in V(T)$, a vertex $v\in B_x$, and some edge $vw$ of the torso $G_x$ that is not in $G[B_x]$. By definition of the torso, $v,w\in B_x\cap B_y$ for some edge $xy\in E(T)$. As shown above, there is an edge $vu$ in $G$ with $u\in V_{y,x}$; let $\phi_x(v,w) := (v,u)$.
Since $u\notin B_x$ and $|B_x\cap B_y|\leq 8k$ (by \cref{CliqueSize}), we have $|\phi_x^{-1}(v,u)|<8k$ (all the elements in the pre-image of $(v,u)$ with respect to $\phi_x$ are of the form $(v,z)$ with $z\in V_x\cap V_y$). Thus $\deg_{G_x}(v) < 8 k\deg_G(v)\leq 8 k \Delta$.
\end{proof}

Let $C_1=\{v_1,\dots,v_k\}$ be a $k$-clique in a graph $G_1$. Let $C_2=\{w_1,\dots,w_k\}$ be a $k$-clique in a graph $G_2$. Let $G$ be the graph obtained from the disjoint union of $G_1$ and $G_2$ by identifying $v_i$ and $w_i$ for $i\in\{1,\dots,k\}$, and possibly deleting some edges in $C_1$ ($=C_2$). Then $G$ is a \defn{clique-sum} of $G_1$ and $G_2$.

The following is a direct consequence of \cref{GMST2}.

\begin{cor}
\label{GMST3}
For every proper minor-closed class\/ $\mathcal{G}$, there exists\/ $k\in\N$ such that every graph\/ $G$ in\/ $\mathcal{G}$ with maximum degree at most\/ $\Delta\in\N$ is obtained by clique-sums of\/ $k$-almost\/$^{\downarrow}\!$ embeddable graphs of maximum degree less than\/ $8k\Delta$.
\end{cor}

\subsection{Partitions}

A \defn{vertex-partition}, or simply \defn{partition}, of a graph $G$ is a set $\PP$ of non-empty sets of vertices in $G$ such that each vertex of $G$ is in exactly one element of $\PP$. Each element of $\PP$ is called a \defn{part}. The \defn{quotient} of $\PP$ is the graph, denoted by $G/\PP$, with vertex set $\PP$ where distinct parts $A,B\in \PP$ are adjacent in $G/\PP$ if and only if some vertex in $A$ is adjacent in $G$ to some vertex in $B$.

A partition $\PP$ of a graph $G$ is called an \defn{$H$-partition} if $H$ is a graph that contains a spanning subgraph isomorphic to the quotient $G/\PP$. Alternatively, an \defn{$H$-partition} of a graph $G$ is a partition $(A_x:x\in V(H))$ of $V(G)$ indexed by the vertices of $H$, such that for every edge $vw\in E(G)$, if $v\in A_x$ and $w\in A_y$ then $x=y$ or $xy\in E(H)$. The \defn{width} of such an $H$-partition is $\max\{|A_x|: x\in V(H)\}$. Note that a layering is equivalent to a path-partition.

\citet*{DJMMUW20} introduced a layered variant of partitions (analogous to layered treewidth being a layered variant of treewidth). The \defn{layered width} of a partition $\PP$ of a graph $G$ is the minimum integer $\ell$ such that for some layering $(V_0,V_1,\dots)$ of $G$, each part in $\PP$ has at most $\ell$ vertices in each layer $V_i$. A partition $\PP$ of a graph $G$ is a \defn{$(k,\ell)$-partition} if $\PP$ has layered width at most $\ell$ and $G/\PP$ has treewidth at most $k$. A class $\mathcal{G}$ of graphs \defn{admits bounded layered partitions} if there exist $k,\ell\in\N$ such that every graph in $\mathcal{G}$ has a $(k,\ell)$-partition.

Several recent results show that various graph classes admit bounded layered partitions. The first results were for minor-closed classes by \citet*{DJMMUW20}, who proved that planar graphs admit bounded layered partitions; more generally, that graphs of bounded Euler genus admit bounded layered partitions; and most generally, a minor-closed class admits bounded layered partitions if and only if it excludes some apex graph. Some results for non-minor-closed classes were recently obtained by \citet*{DMW19b}. For example, they proved that $(g,k)$-planar graphs and $(g,d)$-map graphs admit bounded layered partitions amongst other examples.

\citet*{DJMMUW20} showed that this property implies bounded layered treewidth.

\begin{lem}[\citep{DJMMUW20}]
\label{PartitionLayeredTreewidth}
If a graph\/ $G$ has a\/ $(k,\ell)$-partition, then\/ $G$ has layered treewidth at most\/ $(k+1)\ell$.
\end{lem}

What distinguishes layered partitions from layered treewidth is that layered partitions lead to constant upper bounds on the queue-number and non-repetitive chromatic number, whereas for both these parameters, the best known upper bounds obtainable via layered treewidth are $O(\log n)$. This led to the positive resolution of two old open problems; namely, whether planar graphs have bounded queue-number \citep{DJMMUW20} and whether planar graphs have bounded non-repetitive chromatic number \citep{DEJWW20}. Other applications include $p$-centred colouring \citep{DFMS20} and graph encoding / universal graphs~\citep{BGP20,DEJGMM,EJM}.

Our next tool is the following result by \citet*{DJMMUW20}.

\begin{lem}[\citep{DJMMUW20}]
\label{AlmostEmbeddableStructure}
Every\/ $k$-almost embeddable graph with no apex vertices has an\/  $(11k+10,6k)$-partition.
\end{lem}

\subsection{Excluding a Minor}

We now prove that a result like \cref{AlmostEmbeddableStructure} also holds for $k$-almost embeddable graphs in which all the apex vertices have bounded degree (and in particular if the graph has bounded degree). 

\begin{lem}
\label{DropApices}
Let\/ $G$ be a graph such that, for some\/ $A\subseteq V(G)$, every vertex in\/ $A$ has degree at most\/ $\Delta\in\N$, and\/ $G-A$ has a\/ $(k,\ell)$-partition. Then\/ $G$ has a\/ $(k+1,2\ell\Delta|A|)$-partition.
\end{lem}

\begin{proof}
Let $\PP$ be a $(k,\ell)$-partition of $G-A$, where $\PP$ has layered width at most $\ell$ with respect to a layering $(V_0,V_1,\dots)$ of $G-A$. Let $I$ be the set of integers $i$ such that some vertex in $A$ has a neighbour in $V_i$. Thus $|I| \leq \Delta |A|$. Let $P$ be the path graph $(0,1,\dots)$. For $j\in\NN$, let $d_j$ be the minimum distance in $P$ from $j$ to a vertex in $I$. For $i\in\NN$, let $W_i$ be the union of the sets $V_j$ such that $d_j=i$. For each edge $vw$ of $G$, if $v\in V_a$ and $w\in V_b$ then $|a-b|\leq 1$, implying $|d_a-d_b|\leq 1$. Thus $(W_0,W_1,\dots)$ is a layering of $G-A$.
Observe that each layer $W_i$ is the union of at most $2|I|$ original layers (at most two layers between each pair of consecutive elements in $I$, plus one layer before $\min I$ and one layer after $\max I$).
Thus $\PP$ has layered width at most $2\ell |I|\leq 2\ell\Delta |A|$ with respect to $(W_0,W_1,\dots)$.
By construction, the vertices of $G-A$ that are neighbours of vertices in $A$ are all in $W_0$.
Add $A$ to $W_0$. We thus obtain a layering of $G$.
Let $\QQ$ be the partition of $G$ obtained from $\PP$ by adding one new part $A$. Thus $\QQ$ has layered width at most $2\ell\Delta |A|$ with respect to $(W_0,W_1,\dots)$. Since $G/\QQ$ has only one more vertex than $(G-A)/\PP$, the treewidth of $G/\QQ$ is at most $k+1$.
\end{proof}

\cref{AlmostEmbeddableStructure,DropApices} lead to the next result.

\begin{lem}
\label{AlmostEmbeddableStructureBoundedDegree}
Every\/ $k$-almost\/$^{\downarrow}\!$ embeddable graph\/ $G$ such that every apex vertex has degree at most\/ $\Delta\in\N$ has an\/ $(11k+11,12k^2\Delta)$-partition.
\end{lem}

\begin{proof}
By definition, $G$ is an induced subgraph of a $k$-almost embeddable graph $G'$. Since deleting an apex vertex in a $k$-almost embeddable graph produces another $k$-almost embeddable graph, we may assume that $G$ and $G'$ have the same set $A$ of apex vertices. By \cref{AlmostEmbeddableStructure}, $G'-A$ has an\/ $(11k+10,6k)$-partition $\PP'$. Let $\PP$ be obtained by restricting $\PP'$ to $V(G-A)$. Thus $\PP$ is an\/ $(11k+10,6k)$-partition of $G-A$. Since every vertex in $A$ has degree at most $\Delta$ in $G$, the result follows from \cref{DropApices}.
\end{proof}

\citet*{DJMMUW20} introduced (an equivalent version of) the following definitions and lemmas as a way to handle clique sums. 
Let $C$ be a clique in a graph $G$, and let $\{C_0,C_1\}$ and $\{P_1,\dots,P_c\}$ be partitions of $C$. A $(k,\ell)$-partition $\PP$ of $G$ is \defn{$(C, \{C_0,C_1\}, \{P_1,\dots,P_c\})$-friendly} if $P_1,\dots,P_c\in\PP$ and $\PP$ has layered width at most $\ell$ with respect to some layering $(V_0,V_1,\dots)$ of $G$ with $C_0\subseteq V_0$ and $C_1\subseteq V_1$. 

\begin{lem}[\citep{DJMMUW20}]
	\label{CliqueFriendly}
	Let\/ $G$ be a graph that has a $(k,\ell)$-partition.
	Let\/ $C$ be a clique in\/ $G$, and let\/ $\{C_0,C_1\}$ and\/ $\{P_1,\ldots,P_c\}$ be partitions of\/ $C$ such that\/ $|C_j \cap P_i|\le 2\ell$ for each\/ $j \in \{0,1\}$ and each\/ $i\in \{1,\ldots,c\}$. Then\/ $G$ has a\/ $(C,\{C_0,C_1\},\{P_1,\ldots,P_c\})$-friendly\/ $(k+c,2\ell)$-partition.
\end{lem}

A graph $G$ \defn{admits clique-friendly $(k,\ell)$-partitions} if for every clique $C$ in $G$, and for all partitions $\{C_0,C_1\}$ and $\{P_1,\dots,P_c\}$ of $C$, there is a $(C,\{C_0,C_1\},\{P_1,\dots,P_c\})$-friendly $(k,\ell)$-partition of $G$. A graph class $\mathcal{G}$ \defn{admits clique-friendly $(k,\ell)$-partitions} if every graph in $\mathcal{G}$ admits clique-friendly $(k,\ell)$-partitions. 

\begin{lem}[\citep{DJMMUW20}]
	\label{CliqueFriendlyCliqueSum}
	Let\/ $\cal G$ be a graph class that admits clique-friendly\/ $(k,\ell)$-partitions. Then the class of graphs obtained from clique-sums of graphs in\/ $\cal G$ admits clique-friendly\/ $(k,\ell)$-partitions.
\end{lem}



%

\cref{AlmostEmbeddableStructureBoundedDegree,CliqueFriendly} lead to the next result.

\begin{lem}
\label{CliqueFriendlyAlmostEmbeddable}
Every\/ $k$-almost\/$^{\downarrow}\!$ embeddable graph\/ $G$ of maximum degree at most\/ $\Delta\in\N$ admits clique-friendly\/ $(19k+11,24k^2\Delta)$-partitions.
\end{lem}

\begin{proof}
By \cref{AlmostEmbeddableStructureBoundedDegree}, $G$ has an\/ $(11k+11,12k^2\Delta)$-partition. It follows from \cref{CliqueFriendly,CliqueSize} that $G$ admits clique-friendly $(19k+11,24k^2\Delta)$-partitions. 
\end{proof}

The following result, of independent interest, says that bounded-degree graphs excluding a fixed minor admit bounded layered partitions.

\begin{thm}
\label{MinorFreeDeltaLayeredPartition}
For every fixed graph\/ $H$, there is a constant\/ $k\in\N$ such that every\/ $H$-minor-free graph with maximum degree\/ $\Delta\in\N$ has a\/ $(k,k\Delta)$-partition.
\end{thm}

\begin{proof}
Let $G$ be an $H$-minor-free graph with maximum degree $\Delta$. By \cref{GMST3}, there is a constant $k_0$ (depending only on $H$) such that $G$ can be obtained by clique-sums of $k_0$-almost$^{\downarrow}\!$ embeddable graphs with maximum degree at most $8k_0\Delta$. By \cref{CliqueFriendlyAlmostEmbeddable}, each such graph admits clique-friendly $(19k_0+11,24k_0^2\cdot 8k_0\Delta)$-partitions. It follows from \cref{CliqueFriendlyCliqueSum} that $G$ also admits clique-friendly $(19k_0+11,192k_0^3\Delta )$-partitions. The result follows where  $k:=\max\{19k_0+11,192k_0^3\}$. 
\end{proof}

With these tools, we are now ready to prove the main result of this section.

\begin{thm}\label{3colMinor}
For every fixed graph\/ $H$, every\/ $H$-minor-free graph\/ $G$ with maximum degree\/ $\Delta\in\N$ is\/ $3$-colourable with clustering\/ $O(\Delta^5)$.
\end{thm}

\begin{proof}
Let $G$ be an $H$-minor-free graph with maximum degree $\Delta$. By \cref{MinorFreeDeltaLayeredPartition}, for some constant $k$ (depending only on $H$), $G$ has a $(k,k\Delta)$-partition. \cref{PartitionLayeredTreewidth} implies that $G$ has layered treewidth at most $(k+1)k\Delta$. By \cref{3ColourLayeredTreewidth}, $G$ has a 3-colouring with clustering $O(k^6\Delta^5)$.
\end{proof}

\subsection{Strong Products}

Some of the above structural results can be interpreted in terms of products. The \defn{strong product} of graphs $A$ and $B$, denoted by $A\boxtimes B$, is the graph with vertex set $V(A)\times V(B)$, where distinct vertices $(v,x),(w,y)\in V(A)\times V(B)$ are adjacent if:
\begin{compactitem}
\item $v=w$ and $xy\in E(B)$, or
\item $x=y$ and $vw\in E(A)$, or
\item $vw\in E(A)$ and $xy\in E(B)$.
\end{compactitem}

\cref{2Colour} was proved using the following result by an anonymous referee of the paper by \citet*{DO95} (refined in \citep{Wood09}).

\begin{lem}[\citep{DO95,Wood09}]
\label{DegreeTreewidthStructure}
Every graph with maximum degree\/ $\Delta\in\N$ and treewidth less than\/ $k\in\N$ is a subgraph of\/ $T\boxtimes K_{20k\Delta}$ for some tree\/ $T$.
\end{lem}

\cref{2Colour} follows from \cref{DegreeTreewidthStructure} by first properly 2-colouring $T$ and then colouring each vertex of the graph by the colour of the corresponding vertex of $T$.

The next observation by \citet*{DJMMUW20} follows immediately from the definitions.

\begin{obs}[\citep{DJMMUW20}]
\label{PartitionProduct}
A graph\/ $G$ has an\/ $H$-partition of layered width at most\/ $\ell\in\N$ if and only if\/ $G$ is a subgraph of\/ $H \boxtimes P \boxtimes K_\ell$ for some path\/ $P$.
\end{obs}

\citet*{DJMMUW20} also showed that if one does not care about the exact treewidth bound, then it suffices to consider partitions with layered width 1.

\begin{obs}[\citep{DJMMUW20}]
\label{MakeWidth1}
If a graph\/ $G\subseteq H\boxtimes P\boxtimes K_\ell$ for some graph\/ $H$ of treewidth at most\/ $k$ and for some path\/ $P$, then\/ $G\subseteq H' \boxtimes P$ for some graph\/ $H'$ of treewidth at most\/ $(k+1)\ell-1$.
\end{obs}

By these two observations, \cref{MinorFreeDeltaLayeredPartition} can be restated as follows:

\begin{thm}
\label{MinorFreeDegreeStructure}
For every fixed graph\/ $X$, every\/ $X$-minor-free graph with maximum degree\/ $\Delta\in\N$ is a subgraph of\/ $H\boxtimes P$ for some graph\/ $H$ of treewidth\/ $O(\Delta)$ and for some path\/ $P$.
\end{thm}

It is worth highlighting the similarity of \cref{DegreeTreewidthStructure,MinorFreeDegreeStructure}.
\cref{DegreeTreewidthStructure} says that graphs of bounded treewidth and bounded degree are subgraphs of the product of a tree and a complete graph of bounded size, whereas \cref{MinorFreeDegreeStructure} says that bounded-degree graphs excluding a fixed minor are subgraphs of the product of a bounded treewidth graph and a path.

\section{Open Problem}

We conclude with a natural open problem that arises from this work. Are planar graphs with maximum degree $\Delta$ 3-colourable with clustering $O(\Delta)$? A construction of \citet*{KMRV97} shows a lower bound of $\Omega(\Delta^{1/3})$, while a slightly different construction by \citet*{EJ14} shows $\Omega(\Delta^{1/2})$.

\subsection*{Acknowledgements}
 
This research was initiated at the Graph Theory Workshop held at Bellairs Research Institute in April 2019. Thanks to the other workshop participants for creating a productive working atmosphere.

  \let\oldthebibliography=\thebibliography
  \let\endoldthebibliography=\endthebibliography
  \renewenvironment{thebibliography}[1]{%
    \begin{oldthebibliography}{#1}%
      \setlength{\parskip}{0.0ex}%
      \setlength{\itemsep}{0.0ex}%
  }{\end{oldthebibliography}}


\def\soft#1{\leavevmode\setbox0=\hbox{h}\dimen7=\ht0\advance \dimen7
	by-1ex\relax\if t#1\relax\rlap{\raise.6\dimen7
		\hbox{\kern.3ex\char'47}}#1\relax\else\if T#1\relax
	\rlap{\raise.5\dimen7\hbox{\kern1.3ex\char'47}}#1\relax \else\if
	d#1\relax\rlap{\raise.5\dimen7\hbox{\kern.9ex \char'47}}#1\relax\else\if
	D#1\relax\rlap{\raise.5\dimen7 \hbox{\kern1.4ex\char'47}}#1\relax\else\if
	l#1\relax \rlap{\raise.5\dimen7\hbox{\kern.4ex\char'47}}#1\relax \else\if
	L#1\relax\rlap{\raise.5\dimen7\hbox{\kern.7ex
			\char'47}}#1\relax\else\message{accent \string\soft \space #1 not
		defined!}#1\relax\fi\fi\fi\fi\fi\fi}

\appendix\section{Alternative proof of the key lemma}

This appendix was added after the paper was accepted to \emph{Combinatorics, Probability \& Computing}. Here we give a slightly simpler and slightly stronger proof of \cref{3Colour}, where we only require seven consecutive layers to have bounded treewidth, and in addition one of the three colour classes contains components of size $O(k\Delta)$, instead of $O(k^3\Delta^2)$.
      
\begin{lem}
\label{3ColourA}
Let\/ $G$ be a graph with maximum degree\/ $\Delta\in\N$.
Let\/ $(V_0,V_1,\dots)$ be a layering of\/ $G$ such that\/ $G[\bigcup_{j=0}^{6}V_{i+j}]$ has treewidth less than\/ $k\in\N$ for all\/ $i\in\NN$. Then\/ $G$ is\/ $3$-colourable with clustering\/ $8000 k^3\Delta^2$.
\end{lem}

\begin{proof}
No attempt is made to improve the constant 8000. 
We may assume (by renaming the layers) that $V_0=V_1=V_2=V_3=V_4=\emptyset$. 

As illustrated in \cref{AppendixProofIllustration}, let $H$ be the subgraph of $G$ induced by $\bigcup\{V_i:i\in\N,i\not\equiv 0 \pmod 8\}$. Note that $H$ is the disjoint union of subgraphs of $G$, each induced by at most seven consecutive layers $V_i$. Thus $H$ has treewidth less than $k$. As a subgraph of $G$, $H$ has maximum degree $\Delta$. By \cref{2Colour}, $H$ has a 2-colouring with clustering $20k\Delta$, say with colours blue and yellow. Let $Y$ be the set of yellow vertices in $H$.

\begin{figure}[!b]
	\centering\includegraphics[scale=1.2]{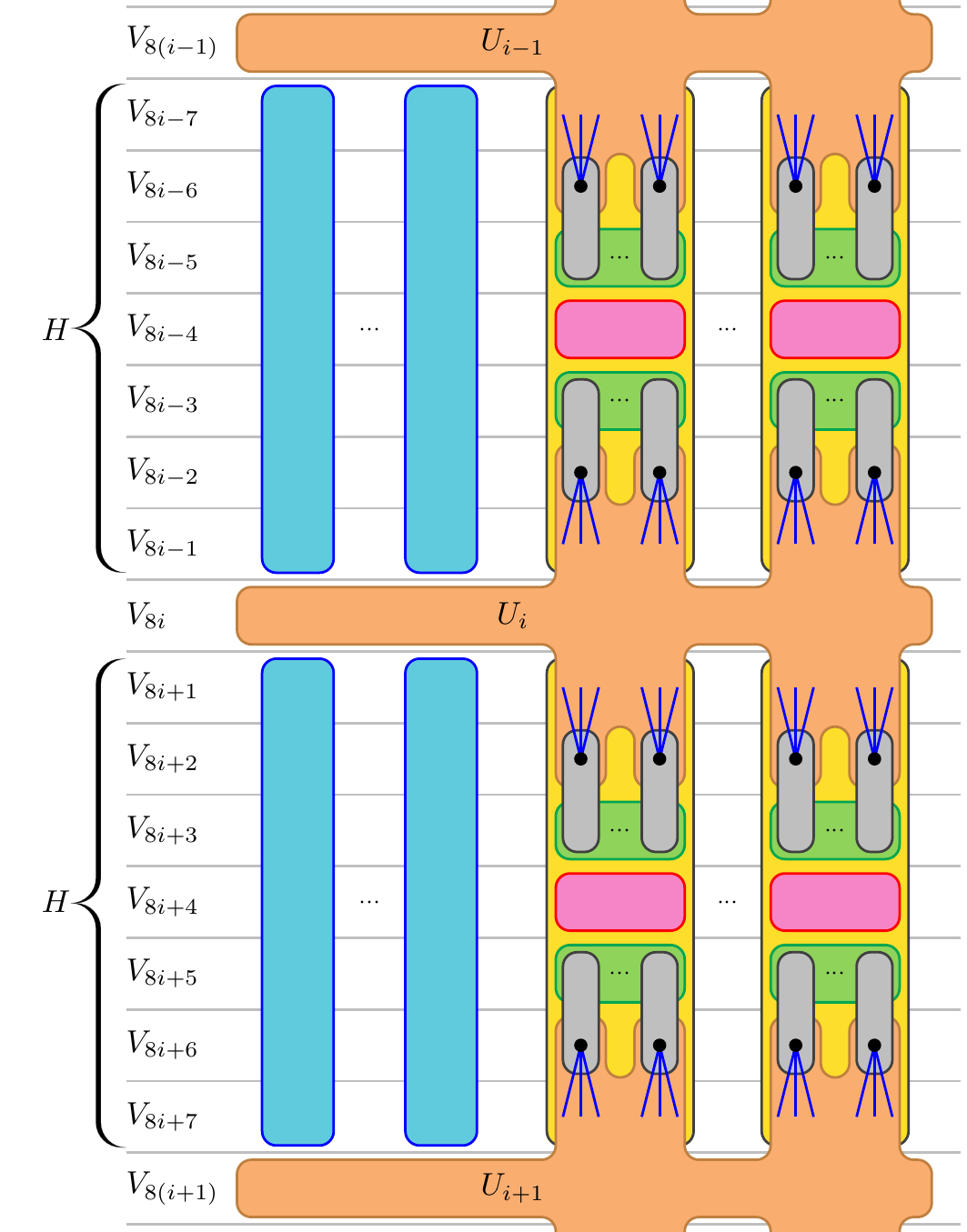}
	\caption{\label{AppendixProofIllustration} Proof of \cref{3ColourA}.}
\end{figure}

We define a colouring $c$ of $G$ as follows. 

Blue vertices in $H$ are coloured blue in $c$, and no other vertex of $G$ will be coloured blue. So each blue monochromatic component has at most $20k\Delta$ vertices. Each non-blue vertex in $G$ will be coloured red or green, giving three colours in total. 

For each $i\in\N$, the vertices in $Y\cap V_{8i+4}$ are coloured red in $c$, and the vertices in $Y\cap ( V_{8i-3} \cup V_{8i+3})$ are coloured green in $c$. This implies that each monochromatic component intersecting $V_{8i+4}$ has at most $20k\Delta$ vertices, and thus each monochromatic component with greater than $20k\Delta$ vertices is red or green and lies in $V_{8i-3}\cup V_{8i-2}\cup \cdots\cup V_{8i+3}$ for some $i\in\N$.

For each $i\in\N$, let $U_i := V_{8i}\cup (Y\cap (V_{8i-2}\cup V_{8i-1}\cup V_{8i+1}\cup V_{8i+2} ))$ and $U_i^+:=U_i \cup (Y\cap (V_{8i-3}\cup V_{8i+3}))$. Let $H_i$ be the graph obtained from $G[U_i^+]$ as follows: contract each connected component $X$ of $G[Y \cap ( V_{8i-3}\cup V_{8i-2} ) ]$ or of $G[Y \cap ( V_{8i+2}\cup V_{8i+3} ) ]$ into a single vertex $v_X$. The neighbours of $v_X$ in $H_i$ lie in $V(X) \cap V_{8i-1}$ or $V(X) \cap V_{8i+1}$. So $v_X$ has degree at most $|V(X)| \leq 20k\Delta$ in $H_i$. Every other vertex $v$ in $H_i$ has degree in $H_i$ at most the degree of $v$ in $G$. So $H_i$ has maximum degree at most $20k\Delta$. Since $H_i$ is a minor of $G[\bigcup_{j=8i-3}^{8i+3}V_j]$ and since treewidth is minor-monotone, $H_i$ has treewidth less than $k$. By \cref{2Colour}, $H_i$ has a colouring $c_i$ with with clustering $20k\cdot 20 k \Delta=400 k^2 \Delta$, say with colours red and green. 

We now define the colouring $c$ of the vertices of $U_i$: For each component $X$ of $G[Y \cap ( V_{8i-3}\cup V_{8i-2}) ]$, assign the colour of $v_X$ in $c_i$ to each vertex in $V(X) \cap V_{8i-2}$. Similarly, for each component $X$ of $G[Y \cap ( V_{8i+3}\cup V_{8i+2} ) ]$, assign the colour of $v_X$ in $c_i$ to each vertex in $V(X) \cap V_{8i+2}$. For every other vertex $v$ of $U_i$, let $c(v):=c_i(v)$. This completes the definition of the colouring $c$ of $G$.

Consider a monochromatic component $M$ in the 3-colouring $c$ of $G$. As shown above, if $|V(M)|>20k\Delta$ then $M$ lies in $U_i^+$ for some $i\in\N$, and $M$ is coloured red or green. By construction, $M$ is contained in the graph obtained from some monochromatic component $M'$ of $H_i$ (with respect to $c_i$) by replacing each contracted vertex $v_X$ in $H_i$ by a subset of $V(X)$. Since $|V(M')| \leq 400 k^2\Delta$ and $|V(X)| \leq 20k\Delta$, we conclude that $|V(M)| \leq 8000 k^3\Delta^2$.
\end{proof}
\end{document}